\documentclass{elsarticle}
\usepackage{graphicx}
\usepackage{subfigure}
\usepackage{color}
\usepackage{newlfont}
\usepackage{multirow}
\usepackage{longtable} 
\usepackage{ifthen}
\usepackage{alltt}
\usepackage{enumerate}
 \usepackage[text={6.25in,8.5in},centering]{geometry} 
\newcommand{\ba}{\begin{array} }
\newcommand{\ea}{\end{array} }
\newcommand{\bae}{\begin{eqnarray}}
\newcommand{\eae}{\end{eqnarray}}
\newcommand{\bea}{\begin{eqnarray*}}
\newcommand{\eea}{\end{eqnarray*}}
\newcommand{\be}{\begin{equation}}
\newcommand{\ee}{\end{equation}}

\newcommand{\pr}{{\bf Proof}~~}

\usepackage{amssymb}
\usepackage{amsthm}
\usepackage{amsmath}

\usepackage{graphicx}
\usepackage{color}
\usepackage[colorlinks]{hyperref}
\usepackage{lscape}
\usepackage{graphicx}
\usepackage{epstopdf}
\newtheorem{theorem}{\hskip\parindent\bf Theorem}[section]
\newtheorem{lemma}{\hskip\parindent\bf Lemma}[section]

\newtheorem{corollary}{\hskip\parindent\bf Corollary}[section]

\begin{document}
 \markboth{Permanence of a general discrete-time two-species-interaction model with non-monotonic per capita growth rate}{Yun Kang}
\title{Permanence of a general discrete-time two-species-interaction model with non-monotonic per capita growth rates}
\author{Yun Kang\footnote{Applied Sciences and Mathematics, Arizona State University, Mesa,
AZ 85212, USA. E-mail: yun.kang@asu.edu}}

\begin{abstract}
Combined with all density-dependent factors, the per capita growth rate of a species may be non-monotonic. One important consequence is that species may suffer from weak Allee effects or strong Allee effects. In this paper, we study the permanence of a discrete-time two-species-interaction model with non-monotonic per capita growth rates for the first time. By using the average Lyapunov functions and extending the ecological concept of the relative nonlinearity, we find a simple sufficient condition for guaranteeing the permanence of systems that can model complicated two-species interactions. The extended relative nonlinearity allows us to fully characterize the effects of nonlinearities in the per capita growth functions with non-monotonicity. These results are illustrated with specific two species competition and predator-prey models of generic forms with non-monotone per capita growth rates.
\end{abstract}

\bigskip
\maketitle

\section{Introduction}
This paper is concerned with the permanence of a discrete-time two-species interaction model with non-monotonic per capita growth rates that can be described by the following equations:
\bae\label{gx}
x_{t+1}\,&=&\,x_t f(x_t,y_t)\\
\label{gy}
y_{t+1}\,&=&\,y_t g(x_t,y_t)\eae
where $x_t$ and $y_t$ denote population densities of species $x$ and $y$ in season $t$ respectively; $f(x,y)$ and $g(x,y)$ are per capita growth rates of these two species, which are nonnegative and twice differentiable in $\mathbb R^2_+$. In addition, at least one of $\frac{\partial f}{\partial x}, \frac{\partial f}{\partial y},\frac{\partial g}{\partial x},\frac{\partial g}{\partial y}$ change signs in $\mathbb R^2_+$, i.e., the per capita growth rates are non-monotonic.

The coexistence of different species in ecological communities has been a main research theme in ecology. For deterministic models, the idea of permanence, which guarantees convergence on an interior attractor from any strictly positive initial conditions, is regarded as a strong form of coexistence. Permanence of dynamical systems has been studied by many researchers using Lyapunov exponents (Shreiber 2000; Garay and Hofbauer 2003; Salceanu and Smith 2009) and average Lyapunov functions (Garay and Hofbauer 2003; Kon 2004; Kang and Chesson 2010) have used Lyapunov exponents and the notions of unsaturated invariant sets (i.e., a compact invariant set that is repelling respect to the system) and measures for Kolmogorov-type systems. Recent study by Kon (2004) uses an average Lyapunov function to show that nonexistence of saturated boundary fixed points is sufficient for permanence under certain convexity and concavity conditions on the per capita growth rates of the species. Kang and Chesson (2010) make use of the concept of the relative nonlinearity (Chesson 1994) to extend Kon's (2004) results beyond convexity and concavity conditions to arbitrary nonlinearities for two dimensional discrete-time competition and prey-predator models. However, Kang and Chesson (2010) as well as Kon (2004) make assumptions on the monotonicity of $f(x,y)$ and $g(x,y)$, e.g., $ \frac{\partial f}{\partial y}$ and $\frac{\partial g}{\partial x}$ does not change the sign. In this article, we will drop this assumption and derive a easy-to-check permanence criterion for a general two-species interaction model with the per capita growth rates being non-monotone by using the theory of average Lyapunov functions and extending the ecological concept of the relative nonlinearity. The original concept of relative nonlinearity is a species-coexistence mechanism that results from different species having different nonlinear responses to competition together with fluctuations in time or space in the intensity of competition (Chesson 1994\& 2000; Kang and Chesson 2010). This extended concept allows us to fully characterize the effects of nonlinearities in the per capita growth functions, which are of major significance in the presence of fluctuated populations and non-monotonic per capita growth rates.

The structure of the rest paper is organized as follows: In section 2, we introduce some basic terminology of semi-dynamical systems and prove the important lemmas that are critical to derive sufficient criterion for the permanence of two-species interaction models with non-monotonic per capita growth rates; In section 3, we give a simple sufficient condition for the permanence of prey-predation models with non-monotonic per capita growth rates; In section 5, we apply our results to specific competition models and prey-predator models with generic with non-monotone per capita growth rates. We conclude with a discussion of the broader implications and prospectus for future work in this area.

\section{Preliminary results}
Let $$X=\{(x, y): x \geq 0, y\geq 0\},\,\,S_x=\{(x,0): x\geq 0\}, \,\,S_y=\{(0, y): y\geq 0\}$$ and $$S=S_x\bigcup S_y,\,\, M=X\setminus
S=\{(x,y): x>0,y>0\}.$$Define a discrete two dimensional dynamical system $H: X \rightarrow X$ where $X$ is a metric space, and denote $H^0(x,y)=(x,y)$ and $H^n(x,y)=(x_n,y_n)$. We call that $H$ is \emph{positively invariant} in $X$, if $H^n(X)\subseteq X$ for any $n\in\mathbb Z_+$.
We call that $H$ is \emph{dissipative}, if there exists $B>0$, for all initial condition $(x,y)\in X$ such that
\[\limsup \max\{x_n,y_n\}\leq B.\]This implies that $H$ has a compact absorbing set in $X$. For convenience, when a system $H$ is dissipative, we will consider $X$ is a compact metric space.

\noindent We call that $H$ is \emph{permanent} in $M$, if there exists two positive numbers $0<b<B$, for all initial condition $(x,y)\in M$ such that
\[b\leq\liminf \min\{x_n,y_n\}\leq\limsup \max\{x_n,y_n\}\leq B.\]This implies that $H$ has a compact interior attractor $B\subset M$ that attracts all the points in $M$. \\

\noindent The boundary equilibrium $(x^*,y^*)\in S$ is \emph{unsaturated} if 
$$ f(x^*,y^*)\geq 1 \mbox{ and } g(x^*,y^*)\geq 1.$$
\noindent The \emph{positive orbit} of $H$ with initial condition $(x_0,y_0)\in X$ is defined as
\[\gamma^+(x_0,y_0)=\{(x_i,y_i): (x_i,y_i)=H^i(x_0,y_0), \mbox{ for all } i\in Z^+\}.\]
The \emph{omega limit set} of $(x,y)\in R^2_+$ is defined as
\[\omega(x,y)=\{(\xi,\eta): H^{t_j}(x,y)\rightarrow (\xi,\eta) \mbox{ for some sequences } t_j \rightarrow \infty\}.\]
The \emph{omega limit set} of a subset $S\subset X$ of $X$ is defined as
\[\omega(S)=\bigcup_{(x,y)\in S} \omega(x,y).\]

The main goal of this paper is to find a sufficient conditions on $f$ and $g$ such that the system \eqref{gx}-\eqref{gy} is permanent in $M$ when $f(x,0), f(0,y), g(x,0)$ and $g(0,y)$ are non-monotonic in $S$, i.e., at least one of the follows change sign in $S$,
$$\frac{\partial f(x,0)}{\partial x},\,\, \frac{\partial f(0,y)}{\partial y},\,\, \frac{\partial g(0,y)}{\partial y},\,\, \frac{\partial g(x,0)}{\partial x}.$$

\subsection{External Lyapunov exponents}
Let $\{(x_i,y_i)\}_{i=0}^{\infty}$ to be the positive orbit with initial conditions $(x_0,y_0)\in S$. Then the average per capita growth rates of species $x,y$ with initial conditions $(x_0,0)\in S_x$ (or $(0,y_0)\in S_y$) after $n-1$  generations can be represented as
\bae\label{ngrowthxx-yx}
r_n^{xx}(x_0,0)\,=\,\frac{\sum_{i=0}^{n-1}\ln f(x_i,0)}{n}\,\,&,&\,\,
r_n^{yx}(x_0,0)\,=\,\frac{\sum_{i=0}^{n-1}\ln g(x_i,0)}{n}\\
\label{ngrowthyy-xy}
r_n^{yy}(0,y_0)\,=\,\frac{\sum_{i=0}^{n-1}\ln g(0,y_i)}{n}\,\,&,&\,\,r_n^{xy}(0,y_0)\,=\,\frac{\sum_{i=0}^{n-1}\ln f(0,y_i)}{n}
\eae
Define $\check{r}^{xx}(x_0,0),\,\,\check{r}^{yx}(x_0,0),\,\,\check{r}^{yy}(0,y_0),\,\,\check{r}^{xy}(0,y_0)$ as the $\limsup$ of the sequences $$\{r_n^{xx}(x_0,0)\}_{n=1}^{\infty},\{r_n^{yx}(x_0,0)\}_{n=1}^{\infty},\{r_n^{yy}(0,y_0)\}_{n=1}^{\infty},\{r_n^{xy}(0,y_0)\}_{n=1}^{\infty}$$  respectively. Moreover, we use $$\bar{r}^{xx}(x_0,y_0),\,\,\bar{r}^{yx}(x_0,0),\,\,\bar{r}^{yy}(x_0,y_0),\,\,\bar{r}^{xy}(0,y_0)$$ instead if their limits actually exist.

Notice that the quantity $\check{r}^{yx}(x_0,0)$ (or $\check{r}_{xy}(0,y_0)$) is the \emph{external Lyapunov exponent} of $S_x$ (or $S_y$), which gives the average invasion speed of the invader $y$ (or $x$) (Rand, Wilson \& McGlade 1994). If both $\check{r}^{yx}(x_0,0)$ and $\check{r}_{xy}(0,y_0)$ are positive for all $x_0\geq 0, y_0\geq 0$, then species $y$ and $x$ are able to coexist. The interesting question is that what kind of conditions on $f(x,y), g(x,y)$ can guarantee this, therefore guarantee the system \eqref{gx}-\eqref{gy} is permanent. In order to answer this question, we assume that $f(x,y)$ and $g(x,y)$ of \eqref{gx}-\eqref{gy} satisfy Condition \textbf{H} in the rest of the article, where \textbf{H} is stated as follows:
\begin{description}
\item \textbf{H}: Both $f(x,y)$ and $g(x,y)$ are strictly positive and twice differentiable in $X\setminus\{(0,0)\}$ with $f(0,0)>1$ and $g(0,0)\geq 0$.  In addition, the following two limits exist and are continuous:
$$\lim_{x\rightarrow 0}f(x,y)=f(0,y)\,\,\mbox{ and } \,\, \lim_{y\rightarrow 0}g(x,y)=g(x,0).$$
\end{description}
\noindent\textbf{Remark:} Condition \textbf{H} implies that the system \eqref{gx}-\eqref{gy} is positively invariant in $X$. Now we need to show the following two lemmas first.
\begin{lemma}\label{singleper}[Bounded population density] Assume that $f(x,y)$ in the system \eqref{gx}-\eqref{gy} satisfies Condition \textbf{H}. If there exists $0< a_\infty<1$ such that 
 $$\lim_{x\rightarrow\infty}\sup_{y\in\mathbb R_+}f(x,y)=a_\infty,$$ then the population density of species $x$ is bounded by some positive constant. Moreover, if $(x_0,y_0)\in S_x$ with $x_0>0$, then $\bar{r}^{xx}(x_0,0)=0.$ \end{lemma}

\begin{proof} Define $a_1(x)=\sup_{y\in\mathbb R_+}f(x,y)$, then the condition $\lim_{x\rightarrow\infty}\sup_{y\in\mathbb R_+}f(x,y)=a_\infty<1$ indicates that for any $\epsilon +a_\infty<1$, there exists a number $L_\epsilon$ large enough such that 
$$a_1(x)<a_\infty+\epsilon<1 \mbox{ for all } x>L_\epsilon.$$
Since  $a_1(x)=\sup_{y\in\mathbb R_+}f(x,y)$, therefore,
$$f(x,y)\leq a_1(x)<a_\infty+\epsilon<1 \mbox{ for all } x>L_\epsilon, y\geq 0.$$
Now if an initial condition of species $x$ is greater than $L_\epsilon$, i.e., $x_0\geq L_\epsilon$, then there exists a positive integer $N$ such that $x_N<L_\epsilon.$ Assume that this is not true, then for any positive integer $n$, we have $x_n \geq L_\epsilon.$ In particular, we have
$$x_{n}=x_{n-1} f(x_{n-1}, y_{n-1})=x_0\prod_{i=0}^{n-1} f(x_i,y_i)<x_0 (a_\infty+\epsilon)^{n}\rightarrow 0\mbox{ as } n\rightarrow \infty.$$ This is a contradiction to the fact that $x_{n}\geq L_\epsilon$ for all $n\in Z_+$. Therefore, there exists a positive integer $N$, such that $x_N<L_\epsilon$.

\indent Define $L_m=L_\epsilon\max_{(x,y)\in [0,L_\epsilon]^2}\{f(x,y)\}.$ 
We claim that if $x_N<L_\epsilon$, then $x_n\leq L_m \mbox{ for all } n>N.$
Suppose that this is not true, then there exists some positive integer $P$ such that $x_{N+p}>L_m.$ Let $ p_m=\min\{p+N: x_{N+p}>L_m\}$, then we have
$$  x_{p_m}>L_m \mbox{ and } x_{p_m-1}\leq L_m.$$
This implies that either $$ x_{p_m-1}\leq L_\epsilon \mbox{ or } L_\epsilon < x_{p_m-1}\leq L_m.$$
If $x_{p_m-1}\leq L_\epsilon$, then $$x_{p_m}=x_{p_m-1}f(x_{p_m-1},y)\leq L_\epsilon\max_{(x,y)\in [0,L_\epsilon]^2}\{f(x,y)\}=L_m$$ which is a contradiction to $x_{p_m}>L_m.$

\noindent If $L_\epsilon<x_{p_m-1}\leq L_m$, then due to the fact that 
$$f(x,y)<a_\infty+\epsilon<1\mbox{ for all } (x,y)\in [L_\epsilon,\infty)\times [0,\infty),$$  we have
$$x_{p_m}=x_{p_m-1}f(x_{p_m-1},y)\leq x_{p_m-1}(a_\infty+\epsilon) <L_m$$ which is also a contradiction to $x_{p_m}>L_m.$ Therefore, we have $$x_n\leq L_m \mbox{ for all } n>N.$$
This implies that for any initial condition $(x_0,y_0)\in X$ with $x_0>0$, there exists a positive integer $p_m$, such that
 $$x_n\leq L_m \mbox{ for all } n>p_m.$$
Therefore, the population density of species $x$ is bounded  in the system \eqref{gx}-\eqref{gy}. 
 
Next, notice that $S_x$ is positively invariant, then for any initial condition in $S_x$, we have $y_n=0$ for all future $n>0$, i.e., we have the following boundary dynamics,
$$x_{n+1}=x_n f(x_n,0) \mbox{ for all } n\geq 0.$$
Then by applying Lemma B.1 (Kang and Chesson 2010), we can conclude that for any initial condition $(x_0,0)$ with $x_0>0$, the following inequalities hold
$$0<b<\liminf_{n\rightarrow\infty}x_n\leq\limsup_{n\rightarrow\infty}x_n\leq L_m.$$
Let $\{(x_i,0)\}_{i=0}^\infty$ to be the positive orbit $\gamma^+(x_0,0)$ starting at $x_0>0$, then we have
 \[0=\lim_{n\rightarrow\infty}\frac{\ln{\frac{b}{x_0}} }{n} \leq \liminf_{n\rightarrow\infty}\frac{\ln{\frac{x_{n-1}}{x_0}}}{n}\leq\limsup_{n\rightarrow\infty}\frac{\ln{\frac{x_{n-1}}{x_0}} }{n}\leq \limsup_{n\rightarrow\infty}\frac{\ln{\frac{B}{x_0}}}{n}=0\]
 This implies that for all $x_0> 0$, we have $$\bar{r}^{xx}(x_0,0)=\lim_{n\rightarrow\infty} \frac{\sum_{i=0}^{n-1}\ln f(x_i,0)}{n}=\lim_{n\rightarrow\infty} \frac{\sum_{i=0}^{n-1}\ln \frac{x_{n-1}}{x_0}}{n}=0.$$Therefore, we have proved the statement. 
\end{proof}
\noindent\textbf{Remark:} Lemma \ref{singleper} gives an easy-to-check sufficient criterion for species $x$ being bounded in a joint system \eqref{gx}-\eqref{gy}, which can not only can apply to competition and prey-predator models but also can apply to the mutualism system. For example, we can apply Lemma \ref{singleper} to show that the following mutualism model \eqref{m1}-\eqref{m2} (Kon 2004) is dissipative:
\bae\label{m1}
x_{t+1}&=&x_t e^{r_1-a_{11}x_t+\frac{a_{12}y_t^{v_{11}}}{1+y_t^{v_{12}}}}\\
\label{m2}
y_{t+1}&=&y_te^{r_2-a_{22}y_t+\frac{a_{21}x_t^{v_{22}}}{1+x_t^{v_{21}}}}
\eae where $0\leq v_{11}\leq v_{12}$ and $0\leq v_{22}\leq v_{21}$.

For convenience, define $F(x,y)=\ln f(x,y), G(x,y)=\ln g(x,y)$ and $F_i, G_i, i=x,y$ as the first partial derivative respect to $i$; $F_{ii}, G_{ii}, i=x,y$ as the second partial derivative respect to $i$.  Let $\{(x_i,0)\}_{i=0}^\infty$ to be a positive orbit $\gamma^+(x_0,0)$ with $x_0\geq 0$. Then we have the following lemma:

\begin{lemma}\label{relative}[The external Lyapunov exponent]
Let $(x^*,0)$ be a point in $S_x$ such that  $f(x^*,0)>0$ and $F_x(x^*,0)\neq 0$. Then the following two cases hold if Condition \textbf{H} holds as  well as $$\lim_{n\rightarrow\infty} f(x,0)=a_1<1.$$ \begin{description}
\item \textbf{Case one:} If $f(0,0)>1$, then $\bar{r}^{xx}(x_0,0)=0$ where $x_0>0$. In addition, the external Lyapunov exponent of $S_x$ is
$$\check{r}^{yx}(x_0,0)=G(x^*,0)-\frac{G_x(x^*,0)}{F_x(x^*,0)}F(x^*,0)+\Delta_y(x_0).$$
\item \textbf{Case two:} If $f(0,0)\leq1$, then $\check{r}^{xx}(x_0,0)\leq 0$ where $x_0>0$. In addition, if $\frac{G_x(x^*,0)}{F_x(x^*,0)}\leq 0$, then the \emph{external Lyapunov exponent} of $S_x$ satisfies follows
 $$\check{r}^{yx}(x_0,0)\geq G(x^*,0)-\frac{G_x(x^*,0)}{F_x(x^*,0)}F(x^*,0)+\Delta_y(x_0)$$   
\end{description}  where $$\Delta_y=\limsup \frac{\sum_{i=0}^{n-1}(x_i-x^*)^2\int^1_0{(1-t)\left[G_{xx}(x_{it},0)-\frac{G_x(x^*,0)F_{xx}(x_{it},0)}{F_{x}(x^*,0)}\right]dt}}{n}$$ and
$x_{it}=x^*+(x_i-x^*)t$.
\end{lemma}
\begin{proof}
The condition that $\lim_{n\rightarrow\infty} f(x,0)=a_1<1$, indicates that the superior of the average growth rate of the species $x$ without species $y$ is less than or equal zero by applying Lemma \ref{singleper}, i.e., for all initial conditions $x_0>0$, we have $$\check{r}^{xx}(x_0,0)=\limsup_{n\rightarrow\infty}\frac{\sum_{i=0}^{n-1} F(x_i,0)}{n}=\limsup_{n\rightarrow\infty}\frac{\ln(x_n/x_0)}{n}\leq \limsup_{n\rightarrow\infty}\frac{\ln(B/x_0)}{n}=0.$$
Doing exact $2^{nd}$ order Taylor expansion on $F(x_i,0), G(x_i,0)$ around $x=x^*$ gives:
 $$\begin{array}{rcl}
F(x_i,0)&=&F(x^*,0)+ F_x(x^*,0)(x_i-x^*)+(x_i-x^*)^2\int^1_0{(1-t)F_{xx}(x_{it},0)dt}\\\
G(x_i,0)&=&G(x^*,0)+ G_x(x^*,0)(x_i-x^*)+(x_i-x^*)^2\int^1_0{(1-t)G_{xx}(x_{it},0)dt}\end{array} $$where $x_{it}=x^*+(x_i-x^*)t$.
Then we have $$\begin{array}{rcl}
r_n^{xx}(x_0,0)&=&\frac{\sum_{i=0}^{n-1}F(x_i\,,0)}{n}\\
&=&F(x^*,0)+ F_x(x^*,0)\frac{\sum_{i=0}^{n-1}(x_i-x^*)}{n}+\frac{\sum_{i=0}^{n-1}(x_i-x^*)^2\int^1_0{(1-t)F_{xx}(x_{it},0)dt}}{n}
\end{array} $$ 
This implies that we have
$$\begin{array}{rcl}
\frac{\sum_{i=0}^{n-1}(x_i-x^*)}{n}&=&\frac{r_n^{xx}(x_0)-F(x^*,0)}{F_x(x^*,0)}-\frac{\sum_{i=0}^{n-1}(x_i-x^*)^2\int^1_0{(1-t)F_{xx}(x_{it},0)dt}}{nF_x(x^*,0)}
\end{array} $$
This implies that we can rewrite $r_n^{yx}(x_0)$ as follows:
$$\begin{array}{rcl}
r_n^{yx}(x_0)&=&\frac{\sum_{i=0}^{n-1}G(x_i\,,0)}{n}\\
&=&G(x^*,0)+ G_x(x^*,0){\frac{\sum_{i=0}^{n-1}(x_i-x^*)}{n}}+\frac{\sum_{i=0}^{n-1}(x_i-x^*)^2\int^1_0{(1-t)G_{xx}(x_{it},0)dt}}{n}\\
&=&G(x^*,0)+ G_x(x^*,0){\left[\frac{r_n^{xx}(x_0,0)-F(x^*,0)}{F_x(x^*,0)}-\frac{\sum_{i=0}^{n-1}(x_i-x^*)^2\int^1_0{(1-t)F_{xx}(x_{it},0)dt}}{nF_x(x^*,0)}\right]}\\&&+\frac{\sum_{i=0}^{n-1}(x_i-x^*)^2\int^1_0{(1-t)G_{xx}(x_{it},0)dt}}{n}\\
&=&G(x^*,0) + \frac{r_n^{xx}(x_0,0)G_x(x^*,0)}{F_x(x^*,0)}-\frac{F(x^*,0)G_x(x^*,0)}{F_x(x^*,0)}\\&&+\frac{\sum_{i=0}^{n-1}(x_i-x^*)^2\int^1_0{(1-t)\left[G_{xx}(x_{it},0)-\frac{G_x(x^*,0)F_{xx}(x_{it},0)}{F_x(x^*,0)}\right]dt}}{n}
\end{array} $$
Define $\Delta_y(x_0)$ as
$$\Delta_y(x_0)=\int^1_0{(1-t)\left[G_{xx}(x_{it},0)-\frac{G_x(x^*,0)F_{xx}(x_{it},0)}{F_x(x^*,0)}\right]dt},$$
then we have \begin{description}
\item \textbf{Case one:} If $f(0,0)>1$, then by applying Lemma \ref{singleper}, we have 
$$\bar{r}^{xx}(x_0,0)=\lim_{n\rightarrow\infty}r_n^{xx}(x_0,0)=0.$$ This indicates that
$$\begin{array}{rcl}
\check{r}^{yx}(x_0,0)&=& G(x^*,0)-\frac{F(x^*,0)G_x(x^*,0)}{F_x(x^*,0)}+\Delta_y(x_0)\\
\end{array} $$
\item \textbf{Case two:} If $\frac{G_x(x^*,0)}{F_x(x^*,0)}\leq 0$, then due to the fact that $$\liminf_{n\rightarrow\infty}r_n^{xx}(x_0,0)\leq\check{r}^{xx}(x_0,0)\leq0,$$we have
$$\begin{array}{rrl}\check{r}^{yx}(x_0,0)&\geq& G(x^*,0)-\frac{F(x^*,0)G_x(x^*,0)}{F_x(x^*,0)}+\liminf_{n\rightarrow\infty}\frac{r_n^{xx}(x_0,0)G_x(x^*,0)}{F_x(x^*,0)}+\Delta_y(x_0)\\
&\geq&G(x^*,0)-\frac{F(x^*,0)G_x(x^*,0)}{F_x(x^*,0)}+\Delta_y(x_0)
\end{array} $$
\end{description}
\end{proof}
\noindent\textbf{Remark:} In the case that $(x^*,0)$ is a nontrivial boundary fixed point, then $F(x^*,0)=0$, thus we have the following corollary from Lemma \ref{relative}:
\begin{corollary}\label{c:relative}[The external Lyapunov exponent] Assume that all conditions in Lemma \ref{relative} hold as well as $F(x^*,0)=0$, then
\begin{description}
\item \textbf{Case one:} If $f(0,0)>1$, then the external Lyapunov exponent of $S_x$ is
$$\check{r}^{yx}(x_0,0)=G(x^*,0)+\Delta_y(x_0).$$
\item \textbf{Case two:} If $f(0,0)\leq1$ and $\frac{G_x(x^*,0)}{F_x(x^*,0)}\leq 0$, then the \emph{external Lyapunov exponent} of $S_x$ satisfies follows
 $$\check{r}^{yx}(x_0,0)\geq G(x^*,0)+\Delta_y(x_0)$$   
\end{description}  where $\Delta_y$ is defined as the same in Lemma \ref{relative}.\end{corollary}
\noindent\textbf{Remark:} The expression $G_{xx}(x,0)-\frac{G_x(x^*,0)F_{xx}(x,0)}{F_{x}(x^*,0)}$ in Corollary \ref{c:relative} can be considered as the extended ecological concept of the relative nonlinearity from Kang and Chesson (2010). This allows us to give an easy-to-check sufficient criterion for permanence applicable to a broad range of situations and avoids checking the detailed information on $\omega(S_x)$ and calculating the external Lyapunov exponent $\check{r}^{yx}(x_0,0)$. Depending on the signs of  $G(x^*,0)$ and $\Delta_y(x_0)$, there are four situations:
\begin{description}
\item \textbf{Permanence}: If both $G(x^*,0)$ and $\Delta_y(x_0)$ are nonnegative and at least one of them is positive for all $x_0\in \mathbb R$, then the \emph{external Lyapunov exponent} $\check{r}^{yx}(x_0,0)$ of $S_x$ is positive. Thus we can apply Theorem 2.2 and its corollary 2.3 of Huston (1984) to show that species $y$ is permanent. We will focus on this case in this article.
\item \textbf{Relative Permanence}: Notice that it is possible that $G(x^*,0)<0$ but $\check{r}^{yx}(x_0,0)$ is still positive for almost every $x_0\in S_x$. This is the case when permanence fails due to the nontrivial boundary equilibrium point $(x^{*},0)$ being saturated, which give a proper setting for the \emph{relative permanence} (Kang 2011a; Kang and Smith 2011b).
\item \textbf{Boundary Attractor}: The case when $\check{r}^{yx}(x_0,0) <0 \mbox{ for all } x_0\in S_x$ represents the case when fluctuations associated with the non-point attractor undermine permanence because then the invasion rate is lower than that predicted by the point attractor $(x^{*},0)$. Thus, the system has no permanence due to existing attractors on the $S_x$. There are many models (Kon 2006; Kang \emph{et al} 2008) presenting this scenario under some proper parameter ranges.
\item\textbf{Multiple Attractors}: The case when $\check{r}^{yx}(x_0,0) >0\mbox{ for a dense set of } x_0\in S_x$ is a mixture of case 2 and 3, which can generate rich dynamics such as riddled basin of attractions (Ashwin \emph{et al} 1996; Ferriere and Gatto 1995; Kon 2006).
\end{description}


\section{Sufficient conditions for the permanence of a general two-species interaction models}
 Let species $x$ and $y$ interact with each other in an ecology community and their population density can be described by the system \eqref{gx}-\eqref{gy}. We are interested in a two-species system \eqref{gx}-\eqref{gy} that satisfies Condition \textbf{H} as well as the following conditions:
\begin{description}
\item\textbf{G1:} There exists $a_1,a_2$ such that 
 $$\lim_{x\rightarrow \infty}\sup_{y\geq0}f(x,y)=a_1<1\mbox{ and }\lim_{y\rightarrow\infty}\sup_{x\geq 0}g(x,y)=a_2<1.$$
\item \textbf{G2:} All the nontrivial boundary fixed points are unsaturated, i.e., if $(x^*,0)$ and $(0,y^*)$ are nontrivial boundary fixed points, then $f(0,y^*)>1\mbox{ and } g(x^*, 0)>1$.
\item \textbf{G3:} There exists some point $(x^*,0)\in S_x$ such that $F_x(x^*,0)\neq 0$ and $G(x^*,0)-\frac{G_x(x^*,0)F(x^*,0)}{F_x(x^*,0)}>0$ and the equality \eqref{G3} holds

 \bae\label{G3}
              r^y(x)&=& G_{xx}(x,0)-\frac{G_x(x^*,0)F_{xx}(x,0)}{F_x(x^*,0)}\geq 0, \mbox{ for all } x>0    
\eae
\item \textbf{G4:} If $g(0,0)>1$, then there exists point $(0,y^*)\in S_y$ such that $G_y(0, y^*)\neq 0$ and $F(0,y^*)-\frac{F_y(0,y^*)G(0,y^*)}{G_y(0,y^*)}>0$ and the inequality \eqref{G4} holds
\bae\label{G4}
                               r^x(y)&=& F_{yy}(0,y)-\frac{F_y(0,y^*)G_{yy}(0,y)}{G_y(0,y^*)}\geq 0,\mbox{ for all } y>0                       \eae
If $g(0,0)<1$, then $\frac{F_y(0,y^*)}{G_y(0,y^*)}<0$ holds in addition to \eqref{G4} holds.
  \end{description}
\noindent Condition \textbf{H} guarantees that the population of species $x$ and $y$ will not drop below 0 and species $x$ will not be too close to the origin $(0,0)$. Condition \textbf{G1} implies that both species $x$ and $y$ suffer from
intra-competition, which drops their per capita fecundity below 1 if their population density is too large regardless of other species' population. The per capita growth rates $f,g$ can be non-monotonic with respect to $x, y$.
Condition \textbf{G2} ensures that both species $x$ and $y$ have nonnegative growth rates at the nontrivial boundary equilibria. Condition \textbf{G3}-\textbf{G4}  guarantee that species $x$ and species $y$ have positive invading speed when their population density are rare. The last condition \textbf{G4} also implies that species $y$ is able to persist even if it is not persistent in its single state (i.e., $g(0,0)<1$). More importantly, \eqref{G3}-\eqref{G4} is an extended concept of the relative nonlinearity that introduced by Chesson (Chesson 2000). Relative nonlinearity is a species-coexistence mechanism that results from different species having different nonlinear responses to competition together with fluctuations in time or space in the intensity of competition (Chesson 1994). The expressions \eqref{G3}-\eqref{G4} can be treated as a general form of relative nonlinearity when $F_x(x,0),G_x(x,0), F_y(0,y),G_y(0,y)$ are non-invertible (Chesson 2000). More generally, we can consider $r^y(x)$ (or $r^x(y)$) as a contribution to the invading speed of species $y$ (or $x$) due to species $x$ (or $y$) has fluctuated population in its single state, e.g., if species $x$ (or $y$) has only point attractors $(x^*_i,0),i=1..,u$ (or $(0,y^*_j),j=1,..,v$), then $r^y(x)=0$(or $r^x(y)=0$). If the species $x$ has non-point attractors (i.e., fluctuated populations), then the contribution $r^y(x)$ can be positive or negative. In this paper, we focus on the case when both $r^y(x)$ and $r^x(y)$ are nonnegative and the main goal of this section is to prove the following theorem:
\begin{theorem}[Sufficient conditions on permanence of two-species models]\label{s_per_c} If the system \eqref{gx}-\eqref{gy} satisfies Condition \textbf{G1-G4} as well as Condition \textbf{H}, then it is permanent in $M$.
\end{theorem}

\begin{proof} 
Since the system \eqref{gx}-\eqref{gy} satisfies Condition \textbf{H}-\textbf{G1}, then according to Lemma \ref{singleper},  there exists a positive number $L$, such that for any initial condition $(x_0,y_0)\in X$,  we have $$\limsup_{n\rightarrow\infty}\max\{x_n,y_n\}\leq L.$$
Therefore, the system is dissipative in $X$. In addition, according to Lemma \ref{singleper}, Condition \textbf{H}-\textbf{G1} also implies  $\bar{r}^{xx}(x_0,0)=0 \mbox{ for all } x_0\geq 0$ and 
the following two cases:
\begin{enumerate}
\item $ \bar{r}^{yy}(0,y_0)=0\mbox{ for all } y_0\geq 0$ if $g(0,0)>1$.
\item $ \check{r}^{yy}(0,y_0)\leq0\mbox{ for all } y_0\geq 0$ if $g(0,0)<1$.
\end{enumerate}
 Let $\{(x_k,0)\}_{k=0}^\infty \mbox{ and } \{(0,y_k)\}_{k=0}^\infty$ be positive orbits with initial conditions $x_0>0$ and $y_0>0$ respectively, and 
$$x_{kt}=x^*+(x_k-x^*)t, \,\,y_{kt}=y^*+(y_k-y^*)t, k=1,\cdots,\infty.$$
Then apply Lemma \ref{relative}, we have 
$$\check{r}^{yx}(x_0,0)=G(x_i^*,0)-\frac{F(x^*,0)G_x(x^*,0)}{F_x(x^*,0)}+\Delta_y(x_0)$$
where$$\Delta_y(x_0)= \limsup \frac{\sum_{k=0}^{n-1}(x_k-x^*)^2\int^1_0{(1-t)\left[G_{xx}(x_{kt},0)-\frac{G_x(x^*,0)F_{xx}(x_{kt},0)}{F_{x}(x^*,0)}\right]dt}}{n}$$ and the following two cases depending on the value of $g(0,0)$,
\begin{enumerate}
\item If $g(0,0)>1$, then  $\check{r}^{xy}(0,y_0)=F(0,y^*)-\frac{F_y(0,y^*)G(0,y^*)}{G_y(0,y^*)}+\Delta_x(y_0)$;
\item If $g(0,0)<1$ and $\frac{F_y(0,y^*)}{G_y(0,y^*)}<0$, then $$\check{r}^{xy}(0,y_0)\geq F(0,y^*)-\frac{F_y(0,y^*)G(0,y^*)}{G_y(0,y^*)}+\Delta_x(y_0).$$
 \end{enumerate}
where $$\Delta_x(y_0)=\limsup \frac{\sum_{i=0}^{n-1}(y_k-y^*)^2\int^1_0{(1-t)\left[F_{yy}(0,y_{kt})-\frac{F_y(0,y^*)G_{yy}(0,y_{kt})}{G_{y}(0,y^*)}\right]dt}}{n}.$$
Then according to Condition \textbf{G2},\textbf{G3} and \textbf{G4}, we have 
\begin{enumerate}
\item If $g(0,0)>1$, then
$$\inf_{x_0\geq0}\check{r}^{yx}(x_0,0)>0\mbox{ and } \inf_{y_0\geq 0}\check{r}^{xy}(0,y_0)>0 .$$
\item If $g(0,0)<1$ and $\frac{F_y(0,y^*)}{G_y(0,y^*)}<0$, then 
$$\inf_{x_0\geq0}\check{r}^{yx}(x_0,0)>0\mbox{ and } \inf_{y_0>0}\check{r}^{xy}(0,y_0)>0 .$$

\end{enumerate}

Now we have the following two cases depending on the value of $g(0,0)$,
\begin{enumerate}
\item If $g(0,0)>1$, then define $P(x,y)=x y$, thus
\bea
\frac{P(x_n,y_n)}{P(x_0,y_0)}&=&\prod_{i=0}^{n-1}{f(x_i,y_i)g(x_i,y_i)}=\prod_{i=0}^{n-1} e^{\left(F(x_i,y_i)+G(x_i,y_i)\right)}\\
&=&e^{\sum_{i=0}^{n-1}{\left(F(x_i,y_i)+G(x_i,y_i)\right)}}=e^{n\left[r^{xx}_n(x_0,y_0)+r^{yy}_n(x_0,y_0)\right]}
\eea
Therefore, we have the following inequalities hold
\bea
\sup_{n\geq 0}\liminf_{(x_0,y_0)\in M\rightarrow (x,0) \in S_x}
\frac{P(x_n,y_n)}{P(x_0,y_0)}&\geq&e^{\left(\bar{r}^{xx}(x,0)+\inf_{x\geq0}\check{r}_{yx}(x,0)\right)}>1\\
\sup_{n\geq 0}\liminf_{(x_0,y_0)\in M\rightarrow (0,y) \in S_y}
\frac{P(x_n,y_n)}{P(x_0,y_0)}&\geq&e^{\left(\bar{r}_{yy}(0,y)+\inf_{y\geq0}\check{r}_{xy}(0,y)\right)}>1
\eea
In addition, for all $(x,y)\in S$, we have
\bea
P(x,y)&=&0
\eea
Then by applying Theorem 2.2 of Hutson (1984) to the system \eqref{gx}-\eqref{gy}, we can conclude that the system is permanent in $M$.
\item If $g(0,0)<1$, then define $P(x,y)= x$, thus
\bea
\frac{P(x_n,y_n)}{P(x_0,y_0)}&=&\prod_{i=0}^{n-1}{f(x_i,y_i)}=\prod_{i=0}^{n-1} e^{F(x_i,y_i)}=e^{\sum_{i=0}^{n-1}{F(x_i,y_i)}}=e^{n\left[r^{xy}_n(x_0,y_0)\right]}
\eea
Therefore, we have the following inequalities hold
\bea
\sup_{n\geq 0}\liminf_{(x_0,y_0)\in M\rightarrow (y,0) \in S_y}
\frac{P(x_n,y_n)}{P(x_0,y_0)}&\geq&e^{\inf_{x\geq0}\check{r}_{xy}(0,y)}>1\eea
In addition, for all $(x,y)\in S_x$, we have
\bea P(x,y)&=&0
\eea
Then by applying Theorem 2.2 of Hutson (1984) to the system \eqref{gx}-\eqref{gy}, we can conclude that species $x$ is permanent. This implies that we can restrict the system in the space $[b, L_m] \times [0, L_m]$ with $b>0$. This allows us to apply Hutson's Theorem 2.2 (1984) again to obtain the permanence of species $y$ in the joint system \eqref{gx}-\eqref{gy} by using the average Lyapunov function $P (x, y) = y$. Thus, the system is permanent in $M$.
\end{enumerate}
\end{proof}
\noindent\textbf{Remark:} If the system \eqref{gx}-\eqref{gy} is discrete version of Lokter-Volterra models, then $F(x,0),G(x,0)$ and $F(0,y),G(0,y)$ are linear functions in $x$ and $y$ respectively. Therefore,  Lokter-Volterra models are permanent if Condition \textbf{G2} holds.  

\begin{corollary}[Alternative version of Condition \textbf{G3}-\textbf{G4}]\label{c2}
Assume that the system \eqref{gx}-\eqref{gy} satisfies Condition \textbf{H}, \textbf{G1-G2} and Condition \textbf{G'3-G'4}, then the system is permanent in $X$,  where Condition \textbf{G'3-G'4} are stated as follows:
\begin{description}
\item\textbf{G'3:}  There exists some point $(x^*,0)\in S_x$ such that $F_x(x^*,0)\neq0,\,G(x^*,0)-\frac{G_x(x^*,0)F(x^*,0)}{F_x(x^*,0)}>0$ and for all $x_i>0$, we have
\be\label{G'3}
                               r^y(x_i)=\int^1_0{(1-t)\left[G_{xx}(x_{it},0)-\frac{G_x(x^*,0)F_{xx}(x_{it},0)}{F_x(x^*,0)}\right]dt}\geq 0, \mbox{ where } x_{it}=x^*+(x_i-x^*)t.
                     \ee
\item\textbf{G'4:} If $g(0,0)>1$, then there exists point $(0,y^*)\in S_y$ such that $G_y(0, y^*)\neq 0,\,F(0,y^*)-\frac{F_y(0,y^*)G(0,y^*)}{G_y(0,y^*)}>0$ and the inequality \eqref{G'4} holds for all $y_i>0$,
 \be\label{G'4}
                    r^x(y_i)= \int^1_0{(1-t)\left[F_{yy}(0,y_{it})-\frac{F_y(0,y^*)G_{yy}(0,y_{it})}{G_y(0,y^*)}\right]dt}\geq 0,\mbox{ where } y_{it}=y^*+(y_i-y^*)t.\ee

If $g(0,0)<1$, then $$\frac{F_y(0,y^*)}{G_y(0,y^*)}<0 \mbox{ and } F(0,y^*)-\frac{F_y(0,y^*)G(0,y^*)}{G_y(0,y^*)}>0$$ for some point $(0,y^*)\in S_y$ in addition to \eqref{G'4} holds.
  \end{description}

\end{corollary} 
\noindent\textbf{Remark:} The proof of Corollary \ref{c2} is similar as the proof of Theorem \ref{s_per_c} and is straightforward, therefore we omit the details. In the case that $(x^*,0)$ and $(0,y^*)$ are nontrivial boundary equilibria points (i.e., $F(x^*,0)=G(0,y^*)=0$), then the conditions $F(0,y^*)-\frac{F_y(0,y^*)G(0,y^*)}{G_y(0,y^*)}>0$ and $G(x^*,0)-\frac{G_x(x^*,0)F(x^*,0)}{F_x(x^*,0)}>0$ can be dropped.

\section{Sufficient conditions for the permanence of a general prey-predator model}
Let species $x$ be a prey and $y$ be the predator of $x$ in an ecology community and their population densities can be described by the discrete system \eqref{gx}-\eqref{gy}. In a prey-predator model, the predator per capita growth rate is an increasing function of prey abundance and its per capita growth rate is commonly negative at zero prey abundance. Thus, the predator cannot survive without the prey, i.e. $g(0,y)<1$ for all $y\geq 0$. We are interested in a prey-predator model \eqref{gx}-\eqref{gy} satisfies Condition \textbf{H} as well as the following conditions:
\begin{description}
\item \textbf{P1:} $\frac{\partial g(x,y)}{\partial x}>0$ and $g(0,y)>0$ for all $(x,y)\in X$.
\item\textbf{P2:} There exists $a_1,a_2$ such that 
 $$\lim_{x\rightarrow \infty}\sup_{y\geq0}f(x,y)=a_1<1$$ and  $$\mbox{ for any } x>0, \lim_{y\rightarrow\infty}g(x,y)=a_2<1.$$
\item \textbf{P3:} All the nontrivial boundary fixed points are unsaturated, i.e., if $(x^*,0)$ is an nontrivial boundary fixed point, then $g(x^*, 0)>1$.
\item \textbf{P4:} There exists some boundary fixed point $(x^*,0)$ such that 
 \bae\label{P4}
                     r^y(x)&=& G_{xx}(x,0)-\frac{G_x(x^*,0)F_{xx}(x,0)}{F_x(x^*,0)}\geq 0, \mbox{ for all } x>0
\eae
or there exists some nontrivial boundary fixed point $(x^*,0)$ such that for all $x_i>0$, we have
\be\label{P'4}
                               r^y(x)=\int^1_0{(1-t)\left[G_{xx}(x_{it},0)-\frac{G_x(x^*,0)F_{xx}(x_{it},0)}{F_x(x^*,0)}\right]dt}\geq 0, \mbox{ where } x_{it}=x^*+(x_i-x^*)t.
                     \ee
  \end{description}
Condition \textbf{P1-P2} and \textbf{H} guarantees that the prey $x$ is permanent without the predator $y$. In\textbf{P1}, the condition $g(0,y) >0\,\, y\geq 0$ implies that the predator has positive per capita growth rate; In \textbf{P2}, the condition $\lim_{y\rightarrow\infty}g(x,y)=a_2<1$ for any given $x > 0$ implies that for any given prey population size, $x$, the finite rate of increase of the predator drops below 1. Such behavior is normally described as predator interference, i.e. predator individuals interact negatively with one another limiting their ability to hunt prey. Condition \textbf{P3-P4} makes sure that predator $y$ is persistent with presence of prey $x$. We want to show the following theorem:
\begin{theorem}\label{sufficient_pp}
If the system \eqref{gx}-\eqref{gy} satisfies Condition \textbf{H, P1, P2, P3, P4}, then it is permanent in $M$.
\end{theorem}
\begin{proof} Since the system \eqref{gx}-\eqref{gy} satisfies Condition \textbf{P1-P2}, then by applying Lemma \ref{singleper}, we know that there exists a positive number $L_m$, such that for any initial condition $(x_0,y_0)\in X$,  we have $$\limsup_{n\rightarrow\infty}x_n\leq L_m \mbox{ and }\bar{r}^{xx}(x_0,0)=0 \mbox{ for all } x_0\geq 0.$$
This implies that for any $(x_0,y_0)\in X$, there exists a positive integer $N$, such that  $$x_n\leq L_m \mbox{ for all } n>N.$$ 
By using the fact that $g(x,y)$ is increasing with respect to $x$, we can conclude that there exists $N$ large enough such that 
$$y_{n+1}=y_n g(x_n,y_n)\leq y_n g(L_m, y_n), \mbox{ for all } n>N.$$
Since  and for any given $x>0$, we have
$$\lim_{y\rightarrow\infty}g(x,y)=a_2<1.$$
Therefore, we can apply Lemma B.1 (Kang and Chesson 2010) to conclude that there exists a positive number $B>L_m$ such that $$\limsup_{n\rightarrow\infty}y_n\leq B.$$ Therefore, the system is dissipative in $M$. 

Now we can show that prey $x$ is persistent in the joint system \eqref{gx}-\eqref{gy} by applying Hutson's Theorem 2.2 and its Corollary 2.3 (1984) with an average Lyapunov function $P (x, y)=x$. Then, we can restrict the system in the space $[b, L_m] \times [0, B]$. This allows us to apply Hutson's Theorem 2.2 (1984) again to obtain permanence of $y$ in the joint system by using the average Lyapunov function $P (x, y) = y$. Therefore, the system is permanent in $M$. Hence, we have shown the statement. 
\end{proof}
\noindent\textbf{Remark:} Note that $\frac{\partial g(x,y)}{\partial x}>0$, thus for any nontrivial boundary fixed point $(x^*,0)$, we have $G_x(x^*,0)>0$, thus there are the following two simplified cases:

\begin{enumerate}
\item If both $F(x,0)$ and $G(0,x)$ are convex (or linear), then $F_x(x^*,0)<0$ and $G(x^*,0)>0$ indicates that the system \eqref{gx}-\eqref{gy} is permanent.
\item If $F(x,0)$ is convex (or linear) and $G(x,0)$ is concave (or linear), then $F_x(x^*,0)>0$ and $G(x^*,0)>0$ indicates that the system \eqref{gx}-\eqref{gy} is permanent.
\end{enumerate}
 \section{Applications}
 In this section, we apply the results to particular competition and prey-predator models. 
 \subsection{A competition model with strong Allee effects in one species}
Multiple stable states occur when more than one type of community can stably persist in a single environmental regime. Simple theoretical analyses predict multiple stable states for single species dynamics via strong Allee effects. Perhaps the most common Allee effect occurs in species subject to predation by a generalist predator with a saturating functional response (Schreiber 2003). The population dynamics of a species suffering Allee effects due to predator saturation can be modeled as 
\bae\label{Predation_Allee}
N_{t+1}&=& N_t e^{r(1-N_t/K)}I(N_t)
\eae where $I(N)=e^{-\frac{m}{1+s N}}$ represents the probability of escaping predation by a predator with a saturating functional response and $m, s$ represents predation intensity and the proportion to the handling time (Hassell \emph{et al.} 1976) respectively. Non-dimensionalizing \eqref{Predation_Allee} by setting $y_t=N_t/K$ and $b=sK$ gives
\begin{equation}\label{predation}y_{t+1}=y_t e^{r(1-y_t)-\frac{m}{1+by_t}}
\end{equation}whose dynamics depends exclusively on the quantities $r$; $b$; and $m$. If  $r<m<\frac{r(b+1)^2}{4b}$ and $b>1$, then \eqref{predation} suffers from strong Allee effects which leads to two positive equilibria that can be expressed as follows:
$$y=\frac{r(b-1)\pm\sqrt{r^2(1+b)^2-4mbr}}{2rb}.$$
Assume that the population dynamics of species $x$ and $y$ can be modeled as follows:
\bae\label{strongAllee_x}
x_{t+1}&=&x_t e^{r_1(1-x_t)-\frac{m_1y_t}{1+b_1y_t}}\\
\label{strongAllee_y}
y_{t+1}&=&y_t e^{r_2(1-y_t)-\frac{m_2}{1+b_2y_t}-a x_t (x_t-c)}
\eae where $r_i, i=1,2$ represents the intrinsic growth rates of species $x$ and $y$ respectively; the term  $e^{-\frac{m_1y_t}{1+b_1y_t}}$ represents that species $x$ suffers the saturated inter-competition from species $y$; the term $e^{-a x (x-c)}$ represents that species $y$ benefits from species $x$ if the abundance of $x$ below the threshold $c$, otherwise, species $y$ suffers inter-competition from $x$. This system \eqref{strongAllee_x}-\eqref{strongAllee_y} models the following two features of the interactions between two species:
\begin{enumerate}
\item Multiple equilibria: the species $y$ suffers strong Allee effects in the absence of species $x$, i.e., $r_2<m_2<\frac{r_2(b_2+1)^2}{4b_2}$ and $b_2>1$.
\item Threshold: If the population of species $x$ is below $c$, then species $y$ benefits from the presence of species $x$; however, if the population of species $x$ is above $c$, then species $y$ suffers from inter-competition from $x$.
\end{enumerate}
It is easy to check that the system \eqref{strongAllee_x}-\eqref{strongAllee_y} satisfies Condition \textbf{H-G1}.  The boundary fixed points of the system are
$$(0,0), (1,0), \left(0,y_1^* \right) \mbox{ and }\left(0,y_2^* \right)$$ where
$$y_1^*=\frac{r_2(b_2-1)-\sqrt{r_2^2(1+b_2)^2-4m_2b_2r_2}}{2r_2b_2}\mbox{ and } y_2^*=\frac{r_2(b_2-1)+\sqrt{r_2^2(1+b_2)^2-4m_2b_2r_2}}{2r_2b_2}.$$
Thus, if $r_2-m_2+a(c-1)>0$ and $r_1-\frac{m_1 y_2^*}{1+b_1y_2^*}>0$, then the system \eqref{strongAllee_x}-\eqref{strongAllee_y}  satisfies Condition \textbf{G2}. Let
$$F(x,y)=r_1(1-x)-\frac{m_1y}{1+b_1y} \,\,\mbox{ and } \,\,G(x,y)=r_2(1-y)-\frac{m_2}{1+b_2y}-a x(x-c).$$
Then $$\begin{array}{llll}
F_x(x,0)=-r_1,&F_{xx}(x,0)=0,&F_y(0,y)=-\frac{m_1}{(1+b_1y)^2}, &F_{yy}(0,y)=\frac{2b_1m_1}{(1+b_1y)^3}\\
G_x(x,0)=a(c-2x),&G_{xx}(x,0)=-2a, &G_y(0,y)=-r_2+\frac{b_2m_2}{(1+b_2y)^2}, &G_{yy}(0,y)=-\frac{2b_2^2m_2}{(1+b_2y)^3}
\end{array}
$$Notice that $$F(0,0)=\ln f(0,0)>0, \,\,G(0,0)=\ln g(0,0)<0,\,\, \frac{F_y(0,0)}{G_y(0,0)}=-\frac{m_1}{b_2m_2-r_2}<0$$ and $$F(0,0)- \frac{F_y(0,0)G(0,0)}{G_y(0,0)}=r_1-\frac{m_1(m_2-r_2)}{b_2m_2-r_2},$$then the system satisfies Condition \textbf{G4} (or \textbf{G'4}) if $r_1>m_1$ and 
$$F_{yy}(0,y)-\frac{F_y(0,0)G_{yy}(0,y)}{G_y(0,0)}=\frac{2b_1m_1}{(1+b_1y)^3}-\frac{2m_1b_2^2m_2}{(b_2m_2-r_2)(1+b_2y)^3}\geq 0.$$This implies that if $$b_1\leq b_2\mbox{ and }m_1(b_2m_2-r_2)>b_2^2m_2,$$then the system satisfies Condition \textbf{G4}. 
In addition, we have 
 $$r^y(x_i)=\int^1_0{(1-t)\left[G_{xx}(x_{it},0)-\frac{G_x(1,0)F_{xx}(x_{it},0)}{F_x(1,0)}\right]}dt=- a.$$
Since $0<x_i<\frac{e^{r_1-1}}{r_1}, i=1,\cdots,\infty$, thus if
$$G(1,0)+(x_i-1)^2r^y(x_i)\geq a(c-1)+r_2-m_2-\max\{a, a\left(\frac{e^{(r_1-1)}}{r_1}-1\right)^2\}$$ holds, then the system satisfies Condition \textbf{G'3}. Therefore, based on the discussion above, we can apply Theorem \ref{s_per_c} and Corollary \ref{c2} to gain the following corollary,
 
\begin{corollary}\label{model_strongAllee}
The system \eqref{strongAllee_x}-\eqref{strongAllee_y} is permanent in $M$, if the following conditions hold
\begin{description}
\item Condition 1: $r_2<m_2<\frac{r_2(b_2+1)^2}{4b_2}$ and $b_2>1$
\item Condition 2: $a(c-1)+r_2-m_2-\max\{a, a\left(\frac{e^{(r_1-1)}}{r_1}-1\right)^2\}>0$ and $r_1>\max\{m_1,\frac{m_1}{b_1}\}$
\item Condition 3: $b_1\leq b_2\mbox{ and }m_1(b_2m_2-r_2)>b_2^2m_2$
\end{description}
\end{corollary}
\noindent\textbf{Remark:} Examples of parameter's values that satisfy conditions in Corollary \ref{model_strongAllee} are$$a=m_2=b_1=1, b_2=2.5>1,m_1=4, r_2=0.85<1, r1=4.1, c>21.$$
\subsection{A competition model with weak Allee effects in both species}
A species suffering weak Allee effects due to predator saturation can be modeled by \eqref{predation} when $r>m$. Then a two species competition model that is subject to weak Allee effects due to predator saturation for both species can be modeled as
\bae\label{weakAllee_x}
x_{t+1}&=&x_t e^{r_1(1-x_t)-\frac{m_1}{1+b_1x_t}-a_1y_t}\\
\label{weakAllee_y}
y_{t+1}&=&y_t e^{r_2(1-y_t)-\frac{m_2}{1+b_2y_t}-a_2 x_t}
\eae where $a_i, i=1,2$ are parameters that measure inter-competition between two species. If $r_i>m_i, i=1,2$, then the boundary fixed points of the system \eqref{weakAllee_x}-\eqref{weakAllee_y} are $$(0,0), (x^*,0) \mbox{ and } (0,y^*)$$ where
$$x^*=\frac{r_1(b_1-1)+\sqrt{r_1^2(1+b_1)^2-4m_1b_1r_1}}{2r_1b_1}\mbox{ and }y^*=\frac{r_2(b_2-1)+\sqrt{r_2^2(1+b_2)^2-4m_2b_2r_2}}{2r_2b_2}.$$
It is easy to check that the system \eqref{strongAllee_x}-\eqref{strongAllee_y} satisfies Condition \textbf{H-G1}. Let
$$F(x,y)=r_1(1-x)-\frac{m_1}{1+b_1x}-a_1y \mbox{ and } G(x,y)=r_2(1-y)-\frac{m_2}{1+b_2y}-a_2 .$$
Then $F(0,0)=\ln f(0,0)>0, G(0,0)=\ln g(0,0)>0$ and $$\begin{array}{llll}
F_x(x,0)=-r_1+\frac{b_1m_1}{(1+b_1x)^2},&F_{xx}(x,0)=-\frac{2b_1^2m_1}{(1+b_1x)^3}
,&F_y(0,y)=-a_1, &F_{yy}(0,y)=0\\
G_x(x,0)=-a_2,&G_{xx}(x,0)=0, &G_y(0,y)=-r_2+\frac{b_2m_2}{(1+b_2y)^2}, &G_{yy}(0,y)=-\frac{2b_2^2m_2}{(1+b_2y)^3}
\end{array}$$
Notice that
$$F_{yy}(0,y)-\frac{F_y(0,y^*)G_{yy}(0,y)}{G_y(0,y^*)}=\frac{2a_1b_2^2m_2}{(r_2-\frac{b_2m_2}{(1+b_2y^*)^2})(1+b_2y)^3}$$ and
$$G_{xx}(x,0)-\frac{G_x(x^*,0)F_{xx}(x,0)}{F_x(x^*,0)}=\frac{2a_2b_1^2m_1}{(r_1-\frac{b_1m_1}{(1+b_1x^*)^2})(1+b_1x)^3},$$
thus, if $r_2>\frac{b_2m_2}{(1+b_2y^*)^2}$ and $r_1>\frac{b_1m_1}{(1+b_1x^*)^2}$, then the system \eqref{weakAllee_x}-\eqref{weakAllee_y} satisfies Condition \textbf{G3-G4}.
In addition, if $F(0,y^*)>0$ and $G(x^*,0)>0$, then the system \eqref{strongAllee_x}-\eqref{weakAllee_y} satisfies Condition \textbf{G2}.Therefore, based on the discussion above, we can apply Theorem \ref{s_per_c} to gain the following corollary,
 
\begin{corollary}\label{model_weakAllee}
The system \eqref{weakAllee_x}-\eqref{weakAllee_y} is permanent in $M$, if the following conditions hold
\begin{description}
\item Condition 1: $r_i>m_i, i=1,2$
\item Condition 2: $F(0,y^*)>0$ and $G(x^*,0)>0$
\item Condition 3: $r_2>\frac{b_2m_2}{(1+b_2y^*)^2}$ and $r_1>\frac{b_1m_1}{(1+b_1x^*)^2}$
\end{description}
\end{corollary}
\noindent\textbf{Remark:} Notice that $F_x(x,0)$ and $G_y(0,y)$ will change signs as $x$ and $y$ increasing to certain thresholds. Thus, permanence theorems in Kang and Chesson (2010) and in Kon (2004) fail while Theorem \ref{s_per_c} applies.
\subsection{A prey-predator model}
One-dimensional heuristic models of single-species population dynamics subject to the weak Allee effect can be represented as \eqref{h_wAllee} (Lewis and Kareiva 1993; Amarasekare 1998a\&1998b; Keitt \emph{et al.} 2001)
\bae\label{h_wAllee}
\frac{dx}{dt}&=&x\left(r-b(x-a)^2\right)
\eae where all the parameters are strictly positive and $r> ba^2$. Then the discrete version of \eqref{h_wAllee} can be represented as
\bae\label{dh_wAllee}
x_{t+1}&=&x_t e^{\left(r-b(x_t-a)^2\right)}
\eae 

Let $x_t$ and $y_t$ represent the population density of prey $x$ and predator $y$ at generation
$t$ respectively. Then a prey-predator model with prey subject to weak Allee effects can be defined as
\bae\label{px}
 x_{t+1}&=& x_t e^{\left(r-b(x_t-a)^2\right)-c_1y_t}\\
 \label{py}
  y_{t+1}&=&y_t e^{c_2 x_t^2-d y_t}
\eae where all the parameters are strictly positive and $r>b a^2$. The fact that all the parameters are nonnegative implies that this prey-predator system \eqref{px}-\eqref{py} satisfies Condition \textbf{H, P1, P2}. Thus, \eqref{px}-\eqref{py} is dissipative. Since $r>b a^2$, then the only nontrivial boundary equilibrium is $$(x^*,0)=(a+\sqrt{r/b},0).$$
Therefore, the system satisfies Condition \textbf{P3} since
$$ g(x^*,0)=e^{c_2 (x^*)^2}>1.$$
 Let
$$F(x,y)=r- b(x-a)^2- c_1y\,\, \mbox{ and }\,\, G(x,y)=c_2 x^2-d y.$$
Then $F(0,0)=\ln f(0,0)=r- b a^2>0, G(0,y)=\ln g(0,y)=- d y\leq 0$ and $$\begin{array}{llll}
F_x(x,0)=-2b(x-a),&F_{xx}(x,0)=-2b
,&F_y(0,y)=-c_1, &F_{yy}(0,y)=0\\
G_x(x,0)=2c_2x,&G_{xx}(x,0)=2c_2, &G_y(0,y)=-d, &G_{yy}(0,y)=0
\end{array}$$
This implies that the system satisfies Condition \textbf{P4} if 
$$r^y(x)=G_{xx}(x,0)-\frac{G_x(x^*,0)}{F_x(x^*,0)} F_{xx}(x,0)=2c_2-\frac{2c_2x^*}{b(a-x^*)}=2c_2(1-\frac{a+\sqrt{r/b}}{\sqrt{rb}})\geq 0$$
Therefore, we have the following corollary by applying Theorem \ref{sufficient_pp}:
\begin{corollary}\label{model_p}
Assume that all the parameters of the system \eqref{px}-\eqref{py} are strictly positive and $r>b a^2.$ Then the system is permanent in $M$, if the following conditions hold
$$1-\frac{a+\sqrt{r/b}}{\sqrt{rb}}\geq 0\,\, (\mbox{   i.e.}, a^2<r(b+1/b-2)).$$\end{corollary}
\noindent\textbf{Remark:} Notice that $F_x(x,0)$  will change signs as $x$ depending on the value of $x$. Thus, permanence theorems in Kang and Chesson (2010) and in Kon (2004) fail while Theorem \ref{sufficient_pp} applies.

 \section{Discussion}
 
 \indent Population per capita growth rate describes the per capita rate of growth of a population, as the factor by which population size increases per year, conventionally given the symbol $\lambda=\frac{N_{t+1}}{N_t}$, or as $r=\ln\lambda$. Therefore, population per capita growth rate is the summary parameter of trends in population density or abundance,  which tells us whether density and abundance are increasing, stable or decreasing, and how fast they are changing (Sibly and Hone 2002).  The per capita growth rate of a population can be broken down into negative density-dependent, density-independent, and positive density-dependent factors (Shreiber 2003). Negative density-dependent factors include resource depletion due to competition (Tilman 1982), environment modification (Jones \emph{et al.} 1997), mutual interference (Arditi and Akcakaya 1990) and cannibalism (Fox 1975). Positive density-dependent factors include predator saturation, cooperative predation or resource defense, increased availability of mates, and conspecific enhancement of reproduction (Courchamp \emph{et al.}1999; Stephens and Sutherland 1999; Stephens \emph{et al.} 1999). Combined with all density-dependent factors, population per capita growth rate can be negative density-dependent in some range of population density and positive density-dependent in some other ranges, i.e., the per capita growth rate of a population may be non-montonic with respect to its population density.

Processing multiple nontrivial equilibria is one important consequence of  a species with non-monotonic per capita growth rate. The possibility that plant and animal populations have multiple positive equilibria has received considerable attention in the ecological literature. Theory and observation indicate that natural multi-species assemblies of plants and animals are likely to possess several different equilibrium points (May 1977). Ecological examples include fish (e.g., Peterman 1977; Spencer and Collie 1997), insects (e.g., Ludwig \emph{et al.} 1978; Kuussaari \emph{et al.} 1998; Solow \emph{et al.} 2003), and phytoplankton (e.g., Beltrami 1989). Moreover, subtidal marine ecosystems in general, and reefs in particular, have several attributes which favor the existence of multiple stable states (Knowlton 1992). Beyond its scientific interest, this possibility also has important implications for the conservation and management of natural systems (Carpenter 2001). In this article, we focus on the permanence of a general two-species interaction model with non-monotonic per capita growth rate for the first time. Our results broaden the applications of permanence for general two-species interaction models.

 Many mathematicians (Fonda, 1988; Freedman \& So 1989; Shreiber 2000; Garay and Hofbauer 2003; Kon 2004; Salceanu and Smith 2009; Kang and Chesson 2010), have studied sufficient conditions on permanence of the dynamical systems. In this paper, we give an easy-to-check sufficient condition for guaranteeing the permanence of systems that can model complicated two-species interactions by using the average Lyapunov functions and extending the ecological concept of the relative nonlinearity.  The extended relative nonlinearity allows us to fully characterize the effects of nonlinearities in the per capita growth functions with non-monotonicity.
These results are illustrated with specific two-species competition and predator-prey models. In particular, we would like to point out that our result can apply to these models while theorems in Kang and Chession (2010) as well as Kon (2004) fails. In addition, our theorem can be easily extended to the mutualism models like \eqref{m1}-\eqref{m2}. Extend our results to higher dimensional discrete-time population models (i.e., food-web dynamics in ecology) will be our future work. 


\section*{Acknowledgements}
The author would like to thank Professor Peter Chesson for his encouragement in applying the concept of relative nonlinearity to study the permanence of ecological population models.

\bibliographystyle{elsarticle-harv}

 \end{document}